\documentclass[12pt,english]{amsart}
\usepackage[utf8]{inputenc}
\usepackage{amsmath,amssymb,amsthm}
\usepackage{xcolor}
\usepackage{hyperref}
\usepackage[
backend=biber,
style=alphabetic,
sorting=ynt
]{biblatex}
\newtheorem{lemma}{Lemma}[section]
\newtheorem{prop}{Proposition}[section]

\newtheorem{thm}{Theorem}[section]

\newtheorem{dfn}{Definition}[section]
\newtheorem{rmk}{Remark}[section]

\usepackage{mathtools}
\usepackage{bbm}

\begin{document}

\title[Mabuchi rays, test configurations and quantization]{Mabuchi rays, test configurations and quantization for toric manifolds}

\author{António Gouveia}
\email{antonioafgouveia@tecnico.ulisboa.pt}
\address{Department of Mathematics and Center for Mathematical Analysis, Geometry and Dynamical Systems, Instituto Superior Técnico, Universidade de Lisboa, 1049-001 Lisboa, Portugal} 

\author{José M. Mour\~ao}
\email{jmourao@tecnico.ulisboa.pt}
\address{Department of Mathematics and Center for Mathematical Analysis, Geometry and Dynamical Systems, Instituto Superior Técnico, Universidade de Lisboa, 1049-001 Lisboa, Portugal} 

\author{João P. Nunes}
\email{jpnunes@math.tecnico.ulisboa.pt}
\address{Department of Mathematics and Center for Mathematical Analysis, Geometry and Dynamical Systems, Instituto Superior Técnico, Universidade de Lisboa, 1049-001 Lisboa, Portugal}

\date{\today}

\begin{abstract}We consider Mabuchi rays of toric K\"ahler structures on symplectic toric manifolds which are associated to toric test configurations and that are generated by convex functions on the moment polytope, $P$, whose second derivative has support given by a compact subset $K< P$. Associated to the test configuration there is a polyhedral decomposition of $P$ whose components are  approximated by the components of $P\setminus K$. Along such Mabuchi rays, the toric complex structure remains unchanged 
on $\mu^{-1}(P\setminus \check K)$, where $\check K$ denotes the interior of $K$. 
At infinite geodesic time, the K\"ahler polarizations along the ray converge to interesting new toric mixed polarizations. 

The quantization in these limit polarizations is given by restrictions of the monomial holomorphic sections of the K\"ahler quantization, for monomials corresponding to integral points in $P\setminus \check K$, and by sections on the fibers of the moment map over the integral points contained in $\check K$, which, along the directions ``parallel" to $K$ are holomorphic and which along the directions ``transverse" to $K$ are distributional. These quantizations correspond to quantizations of the central fiber of the test family, in the symplectic picture. 

We present the case of  $S^2$ in detail and then generalize to higher dimensional symplectic toric manifolds. Metrically, at infinite Mabuchi geodesic time, the sphere  decomposes into two discs and a collection of cylinders, separated by infinitely long lines. Correspondingly, the quantization in the limit polarization decomposes into a direct sum of the contributions from the  quantizations of each of these components.

\end{abstract}

\maketitle

\setcounter{tocdepth}{3}
\tableofcontents

\section{Introduction}

The analytic continuation to imaginary time of Hamiltonian symplectomorphisms plays an important role in K\"ahler geometry and in geometric quantization.  In K\"ahler geometry these ``complexified symplectomorphisms" describe geodesics with respect to the Mabuchi metric on the space of K\"ahler metrics \cite{mabuchi,donaldson-sym,phong-sturm,imrn}. In geometric quantization, Mabuchi geodesic rays provide paths of K\"ahler polarizations which, often, connect to interesting real polarizations at infinite geodesic time. By lifting these paths to the bundle of quantum Hilbert spaces of polarized sections one relates different quantizations of the underlying symplectic manifold.

For the cotangent bundle of a compact Lie group the vertical polarization can be connected by a Mabuchi ray to a rich mixed polarization called the Kirwin-Wu polarization, which allows one to relate the Peter-Weyl and Borel-Weil theorems \cite{kmn13,bhkmn}. The case of symmetric spaces of compact type is studied in \cite{bfhmn}.

Mabuchi rays on symplectic toric manifolds can be simply and effectively described in terms of convex functions on the moment polytope $P$ via the symplectic potentials of Guillemin-Abreu \cite{guillemin,abreu}. These have been used to relate the quantization of a symplectic toric manifold in a toric K\"ahler structure with the quantization in the real toric polarization. Quantization in the toric real polarization is given by distributional sections supported on the fibers of the moment map corresponding to the integral points of $P$. In \cite{bfmn} it was show that $L^1$-normalized holomorphic monomial sections converge, along the Mabuchi ray, to the expected dirac delta distributional sections in the quantization of the toric real polarization. If the half-form correction is included one obtains a similar convergence for $L^2$-normalized monomial sections \cite{kmn}. This convergence can be understood in terms of a generalized coherent state transform
which is asymptotically unitary \cite{kmn2}. Generalizations to limit toric polarizations of mixed type have been considered in \cite{P22} and to manifolds with an Hamiltonian torus symmetry in \cite{W22,CLW23,CLW23-2}. Applications to flag varieties and more general integrable systems were considered in \cite{hamilton.konno:2014, , harada.kaveh:2015, hamilton.harada.kaveh:2021}. 

On previous studies, the Mabuchi rays that were considered were generated by a strictly convex function on $P$, so that at infinite geodesic time, one obtains the real toric polarization along the interior $\check P$ of the moment polytope $P$. Metrically, at infinite geodesic time one obtains a Gromov-Hausdorff convergence to the moment polytope $P$ equipped with an Hessian metric \cite{bfmn}.

Let $M$ be a toric symplectic manifold with moment polytope $P$ and moment map $\mu:M\to P.$
In the present work, we generalize previous studies by considering, first on $\mathbb{CP}^1$ in Section \ref{chapter:new; polarizations; supp not 1} and 
then more generally on higher dimensional symplectic toric manifolds in Section \ref{sec_higherdim}, Mabuchi rays associated to test configurations which are generated by a symplectic potential whose second derivative has support on a strictly smaller compact subset $K< P.$ Accordingly, the initial polarization remains unchanged on $\mu^{-1}(P\setminus \check K)$ where $\check K\subset K$ is the interior of $K$. Along $\check K$ the K\"ahler polarizations converge at infinite geodesic time to  mixed  toric polarizations. The lift of the geodesic path of K\"ahler structures to the bundle of polarized quantum Hilbert spaces has remarkably natural features. Holomorphic monomial sections corresponding to integral points in $\check K$ converge to  sections which are distributional along some directions and remain holomorphic along the others. On the other hand, monomial sections for integral points on 
$P\setminus \check K$ just restrict to the appropriate component of $\mu^{-1} (P\setminus K ).$ We describe this convergence both for $L^1$-normalized sections and for half-form corrected quantization in terms of a generalized coherent state transform. 

From the metric point of view, at infinite geodesic time, in the case of $S^2$, the symplectic sphere decomposes into two discs and a collection of cylinders, labelled by the connected components of $\mu^{-1}(P\setminus K)$ and a collections of infinitely long lines, labelled by the connected components of $K$. The quantization in the limit polarization then corresponds to a direct sum of the contributions of each of these 
constituents. This description generalizes to the higher dimensional case.

{}
\section{Preliminaries}\label{section-prelim}

\subsection{Toric K\"ahler structures}
\label{subsec_torickahler}

Consider now $(M,\omega)$ a compact toric manifold, of real dimension $2n$, with moment map $\mu$ and moment polytope $P$, such that the symplectic form satisfies 
$\left[\frac{\omega}{2\pi}\right]-\frac{c_1(M)}{2}\in H^2(M,\mathbb{Z})$. This choice allows us to choose $P$ to be 
\begin{equation}
    P=\{x\in \mathbb{R}^n\,;\,l_j(x)\geq 0,\,\,j=1,...,r\},
\end{equation}
where $\ell_j(x)=\langle x,v_j\rangle-\lambda_r,$ $v_j$ are the primitive vectors normal to the $j$-th facet of the polytope that point to the inside, and $\lambda_j\in 1/2+\mathbb{Z}$.

Let $\check P$ be the interior of $P$. Then we have that $\check M:=\mu^{-1}(\check P)$ is an open dense subset of $M$ consisting of the points in which the action is free. Consider action-angle coordinates $(x,\theta)$ on $\check M\cong \check P \times \mathbb{T}^n$, under which  the moment map becomes $\mu(x,\theta)=x.$ Following the work of Guillemin and Abreu \cite{guillemin, abreu},  one considers the following smooth function $g_P:\check P\rightarrow\mathbb{R}$.
\begin{equation}\label{forabreu2}
	g_P(x)=\sum_{r=1}^d\frac{1}{2}\ell_r(x)\log(\ell_r(x)).
\end{equation}
\begin{thm}[Abreu, \cite{abreu}]
	Let $(M , \omega )$ be the toric symplectic manifold associated to
	a Delzant polytope $P\subset\mathbb{R}^n$, and let 
	$J$ any compatible toric complex structure. Then $J$
	is determined
	by a “symplectic potential” $g\in C^\infty(\check P)$ of the form
	$g = g_P + h$,
	where $g_P$ is given by (\ref{forabreu2}), h is smooth on the whole $P$, and the matrix $G= \text{Hess}_x(g)$ is positive definite on $\check P$
	and has determinant of the form
	\[\displaystyle \det(G) = \left[\delta(x)\prod_{r=1}^d\ell_r(x)\right]^{-1},\]
	with $\delta$ being a smooth and strictly positive function on the whole P.
	Conversely, any such $g$ determines a compatible toric complex structure $J$ on
	$(M,\omega)$.
 \end{thm}
	
The symplectic potential allows us to define an equivariant biholomorphism between $\check P\times\mathbb{T}^n$ and $(\mathbb{C}^*)^n$ through a Legendre transform, mapping $x\in \check P$ to $y:=\frac{\partial g}{\partial x}\in\mathbb{R}^n$, so that 
\begin{equation}\label{holcoor}
w=(w_1,\dots,w_n):=(e^{y_1+i\theta_1},\dots,e^{y_n+i\theta_n}),    
\end{equation}

gives a system of holomorphic coordinates on $\check M$. The inverse Legendre transform is given by 
\[x:=\frac{\partial h}{\partial y},\] where $h(y)=x(y)\cdot y-g(x(y))$ is the K\"ahler potential on $\check M$.
(In an analogous way, as seen, for example, in \cite{kmn}, we can define holomorphic coordinate systems around the vertices of $P$.)

There is a bijective correspondence between the irreducible torus-invariant divisors and the $1$-cones of the fan associated with the toric manifold. These $1$-cones are then given by the primitive integral vectors $\nu_j$ which are normal to the facets of $P$ and it follows that the irreducible divisors are:
\[D_j=\mu^{-1}\left(\{x\in P\,;\,l_j(x):=\langle \nu_j,x\rangle+\lambda_j=0\}\right), j=1, \dots d.\]
Given a divisor $D^L=\sum_{j=1}^r\lambda^L_jD_j$, $\lambda^L_j\in\mathbb{Z}$, we denote the corresponding line bundle by $L=\mathcal{O}(D^L)$. Let $\sigma_{D^L}$ be the unique up to a constant meromorphic section of $L$, with divisor $D^L$. Following Proposition 4.1.2 in \cite{cls}, we obtain that for any meromorphic function $w^m,\,\,m\in\mathbb{Z}^n$, its divisor is given by
\[\text{div}(w^m)=\sum_{j=1}^r\langle \nu_j,m\rangle D_j,\]
and so the space of holomorphic sections is
\begin{align*}
	H^0(M,L)=\text{span}_{\mathbb{C}}\{w^m\sigma_{D^L}\,;\,m\in\mathbb{Z}^n,\,\,\text{div}(w^m_0\sigma_{D^L})=\langle \nu_j,m\rangle+\lambda^L_i\geq0\}.
\end{align*}
There is then a bijection between a (monomial) basis of $H^0(M,L)$ and the integral points of $P$.

\subsection{Mabuchi rays and quantization of toric manifolds}
\label{subsec_quanttoric}

The families of symplectic potentials 
\begin{equation}
    g_s=g_P+s\psi,\,\,s>0,
\end{equation}
where $\psi$ is a strictly convex smooth function on $P$, 
were considered in \cite{bfmn}, \cite{kmn} and \cite{kmn2}, as well as 
the corresponding  families of  holomorphic polarizations and the associated Hilbert space of polarized sections. In \cite{bfmn} $L^1$-normalized sections were studied; in \cite{kmn} half-form corrected $L^2$-normalized sections were considered; in \cite{kmn2} half-form corrected sections were studied by means of Hamiltonian flows in imaginary time and generalized coherent state transforms.

Consider the holomorphic polarizations corresponding to the complex structure $J_s, s\geq 0$, determined on $(M,\omega)$ by the symplectic potential $g_s$,
\begin{equation}\label{Ps}
	\mathcal{P}_s=\text{span}_{\mathbb{C}}\left\{ \frac{\partial}{\partial w^i_s},\,\,i=1,...,n\right\},
\end{equation}
where the holomorphic coordinates $w^j_s, j=1\dots , n,$ are defined by $g_s$ as in (\ref{holcoor}).
One can take the limit as $s\to +\infty$ in the positive Lagrangian Grassmannian of $T_pM\otimes\mathbb{C},$ pointwise for $p\in M$, obtaining a limit polarization
\begin{equation}\label{Pinfinity}
	\mathcal{P}_\infty:=\lim_{s\to\infty}\mathcal{P}_s.
\end{equation}

Let $\mathcal{P}_\mathbb{R}$ be the real toric polarization defined by 
\begin{equation}\label{PR}
	\mathcal{P}_\mathbb{R}=\text{span}_{\mathbb{R}}\left\{ \frac{\partial}{\partial \theta^i},\,\,i=1,...,n\right\}.
\end{equation}

One obtains
{}
\begin{thm}[\cite{bfmn}]\label{funções convergem}
	\[C^\infty\left(\lim_{s\to\infty}\mathcal{P}_s\right)=C^\infty(\mathcal{P}_\mathbb{R}).\]
\end{thm}
Thus, along the above Mabuchi ray of toric K\"ahler structures the holomorphic polarizations converge to the real toric polarization as $s\to\infty$.

The norm of the  holomorphic section associated to the integral point $m\in P\cap \mathbb{Z}^n$, along  $\mu^{-1}(x), x\in \check P$, is given in terms of the functions
\[f_m(x)=(m-x)\frac{\partial \psi}{\partial x}-\psi(x).\]
For $\psi$ strictly convex, this function has a global minimum at $m$, which yields the following result
\begin{prop}[\cite{bfmn}]
	\[\frac{e^{-sf_m(x)}}{||e^{-sf_m}||_1}\xrightarrow{s\to\infty}\delta(x-m),\text{ in the sense of distributions.}\]
\end{prop}
Consider the case when $\left[\frac{\omega}{2\pi}\right]\in H^2(M,\mathbb{Z})$, so that $P$ has integral vertices. Using the Liouville measure, one considers the injection of smooth
in distributional sections:
\begin{align*}
	\iota:C^\infty(L_\omega|_U)&\rightarrow C^{-\infty}(L_\omega|_U):=(C^\infty_c(L^{-1}_\omega|_U))^*\\
	s&\mapsto\iota s(\phi)=\int_{U}s\phi\frac{\omega^n}{n!},
\end{align*}
where $U$ is any open subset of $\check {M}$ and $L_\omega$ is the prequantum line bundle on $\check P$. Following \cite{bfmn}, one defines
\[\delta^m(\tau)=\int_{\mu_P^{-1}(m)}e^{i\ell(m)\theta}\tau=\hat\tau(x=m,\theta=-m),\,\,\forall \tau\in C^\infty_c(L^{-1}_\omega|_U),\]
where $\hat\tau$ represents the Fourier transform of $\tau$. The $J_s$-holomorphic sections can be written as
\[\sigma^m_s:=e^{-h_s(x)}w^m_s\mathbbm{1
},\,\,m\in \check P\cap \mathbb{Z},\]
where $\mathbbm{1}$ is a meromorphic section which gives a trivializing frame along $\check M.$ One has similar expressions for monomial sections correponding to integral points along $\partial P.$
Then, 
\begin{thm}[\cite{bfmn}]\label{important3}
	For $m\in P\cap \mathbb{Z},$ consider the family of $L^1$-normalized $J_s$-holomorphic sections
	\[\mathbb{R}^+\ni s\mapsto \xi_s^m:=\frac{\sigma^m_s}{||\sigma^m_s||_1}\in C^\infty(L_\omega)\xhookrightarrow{\iota}(C^\infty_c(L^{-1}_\omega|_U))^*,\] for $U\subset M$ open.
	Then, as $s\to\infty,$ $\iota(\xi^n_s)$ converges to $\delta^m $ in $(C^\infty_c(L^{-1}_\omega|_U))^*$.
\end{thm}

To include the half-form correction, following \cite{kmn}, we now take $\left[\frac{\omega}{2\pi}\right]-\frac{c_1(M)}{2}\in H^2(M,\mathbb{Z})$. Let $\mathcal{K}_P$ denote the canonical bundle for a toric K\"ahler polarization 
$\mathcal{P}$. As described in \cite{kmn}, there is a natural factorization $\mathcal{K}_P\cong |\mathcal{K}_P|\otimes \mathcal{K}_P^{U(1)}$, where the second factor is always the same for any toric K\"ahler polarization on $M$. As $|\mathcal{K}_P|$ is trivial it admits a square root, denoted by $|\mathcal{K}_P|^{\frac{1}{2}}$.
One defines the half-density $\sqrt{|dX|}(X_1,...,X_n)=|dX(X_1,...,X_n)|^{\frac{1}{2}}$, where $dX = dx_1 \wedge \cdots \wedge dx_n$ and $X_j, j=1, \dots, n$, are vector fields. We define $dZ_s:=dz^1_s\wedge...\wedge dz^n_s$, where $z^j_s = \frac{\partial g_s}{\partial x_j} + i \theta_j, j=1, \dots, n$.
%, as the generators of the fibers defined by the $g_s$. 
A global trivializing section of $|\mathcal{P}_s|^{\frac{1}{2}}$ is then $\frac{\sqrt{|dZ_s|}}{||dZ_s||^{\frac{1}{2}}}$. For each $s\geq 0$ one obtains the Hilbert space of half-form corrected sections polarized with respect to the K\"ahler polarizations $\mathcal{P}_s$ defined by the symplectic potential $g_s$, 
\[\mathcal{H}_s:=\left\{ \sigma\otimes\frac{\sqrt{|dZ_s|}}{||dZ_s||^{\frac{1}{2}}};\,\,\sigma\text{ is a polarized section of }L\right\},\]
where the space of polarized sections of $L$ is generated by the set \begin{equation}\label{half-form-sections}
    \sigma^m_s:=e^{-h_s(x)}w^m_s\mathbbm{1
}^{U(1)},\,\,m\in P\cap \mathbb{Z}.
\end{equation}
\begin{thm}[\cite{kmn}]\label{limittilde}
	\[\frac{\tilde\sigma^m_s}{||\sigma^m_s||_2}\xrightarrow{s\to\infty}2^\frac{n}{2}\pi^\frac{n}{4}\delta^m\otimes\sqrt{|dX|}, \text{ where }\tilde\sigma^m_s=\sigma^m_s\otimes \frac{\sqrt{|dZ_s|}}{||dZ_s||^{\frac{1}{2}}}.\]
\end{thm}
It was also shown in \cite{kmn} that the Hilbert space of half-form corrected sections which are polarized with respect to the real toric polarization is given by 
$$\mathcal{H}_{\mathbb{R}}= \mathrm{span}_{\mathbb{C}}\{\delta^m\otimes \sqrt{|dX|}, m\in P\cap \mathbb{Z}\}.$$ In this way, one describes the convergence of $L^2$-normalized half-form corrected holomorphic sections along the Mabuchi ray of complex structures $J_s, s\geq 0$ to the distributional sections corresponding to the real toric polarization.

{}

This result may reformulated using Hamiltonian flows analytically continued to imaginary time and so-called generalized coherent state transforms \cite{kmn2}. Let $\psi$ be the strongly convex function on $P$ and consider the family of symplectic potentials $g_s = g_P + s\psi, s\geq 0$. Using $z_s^j=\frac{\partial g_s}{\partial x_j}+i\theta_j$, let $\mathcal{P}_s$ be K\"ahler polarization of 
$(M,\omega)$ given by
$$\mathcal{P}_s=\text{span}_\mathbb{C}\left\{\frac{\partial}{\partial {z}_j}, j=1,...,n\right\},$$
\begin{prop}[\cite{kmn2}]\label{thecomplex1}
	Let $s>0$. Then:
	\begin{itemize}
		\item[i)] As distributions, $\mathcal{P}_s=e^{is\mathcal{L}_{X_\psi}}\mathcal{P}_g,$
		\item[ii)] Pointwise, $dZ_s=e^{is\mathcal{L}_{X_\psi}}dZ_0.$	
	\end{itemize}
\end{prop}

Consider now the {Kostant-Souriau prequantum operator} associated to the smooth function $\psi$,  
\[\hat\psi=i\nabla_{X_\psi}+\psi=iX_\psi-x\cdot\frac{\partial\psi}{\partial x}+\psi,\]
where the prequantum connection can be written as $\nabla = d + i \sum_{j=1}^n x_j d\theta_j$
along $\check M$. 

{}

\begin{prop}[\cite{kmn2}]\label{cst1}
	For any $s>0$, the operator $e^{s\hat\psi}\otimes e^{is\mathcal{L}_\psi}:\mathcal{H}_0\rightarrow\mathcal{H}_s$ is a linear isomorphism and
	\[e^{s\hat\psi}\otimes e^{is\mathcal{L}_\psi}\sigma_0^m=\sigma^m_s,\]for all $m\in P\cap \mathbb{Z}$.
\end{prop}
There is a natural quantization of $\psi$  given by the operator:
\begin{align*}
	Q(\psi):\mathcal{H}_s&\rightarrow\mathcal{H}_s,\quad \quad
	\sigma_s^m \otimes \sqrt{dZ_s}\mapsto\psi(m)\sigma_s^m \otimes \sqrt{dZ_s}, m\in P\cap \mathbb{Z}.
\end{align*}
This allows for the definition of  the generalized coherent state transform (gCST) as  
$$A^\psi_{s}:\mathcal{H}_0\rightarrow\mathcal{H}_s,$$
\[ A^\psi_{s}:=\left(e^{s\hat\psi}\otimes e^{is\mathcal{L}_\psi}\right)\circ e^{-sQ(\psi)} \]
One obtains 
\begin{thm}[\cite{kmn2}]\label{compleximportanttheorem}
	\[\lim_{s\to\infty}A^\psi_{s} \left(\sigma_0^m\otimes \sqrt{dZ_s}\right)= (2\pi)^{n/2} e^{g(m)} \delta^m\otimes\sqrt{dX}, \, \, m\in P\cap \mathbb{Z}.
\]
\end{thm}

\subsection{Test configurations}
\label{sec_testconfig}

In \cite{donaldson02}, in the study of $K$-stability for toric varieties, Donaldson describes degenerations of a complex $n$-dimensional smooth toric variety $M$ with moment polytope $P$ by considering rational piecewise linear symplectic potentials on $P$. (See also \cite{chentang08, sz2012}.) For such a potential $f$, he considers a real 
$(n+1)$-dimensional polytope $\hat P$ defined by
$$
\hat P = \left\{ (x,y)\in P\times \mathbb{R} \,\vert\, x\in P, 0\leq y \leq K-f\right\},
$$
where $K=\mathrm{max} (f).$ Up to rescaling, one can take $\hat P$ to be integral and to define a toric variety $\hat M$, together with an holomorphic line bundle $\hat L\to \hat M$,  which defines a ``test configuration" (Section 2 in \cite{donaldson02}). The inclusion $P\times \left\{0\right\}\subset \hat P$ gives an inclusion $M\subset \hat M$ and $\hat M$ comes equipped with a natural 
map $\pi: \hat M \to \mathbb{CP}^1$ which is equivariant with respect to the $\mathbb{C}^*$-action 
on $\hat M$ coming from the last $\mathbb{C}^*$ factor in the $(\mathbb{C}^*)^{n+1}$ toric action. The central fiber of $\hat M\to \mathbb{CP}^1$ corresponds to the polytope 
$$Q= \left\{(x,y)\in P\times \mathbb{R}: x\in P, y=K-f(x)\right\}   \, . 
$$

In \cite{chentang08, sz2012}, the authors study, for compact toric K\"ahler manifolds, the existence of Mabuchi geodesic rays associated to algebraic degenerations defined by test configurations, obtaining

\begin{thm}[Thm. 7.1 \cite{chentang08}, Thm 7 \cite{sz2012}] Let $\hat M$ be a toric degeneration defined by the piecewise linear symplectic potential $f$ on the Delzant polytope $P$ such that the polytope $\hat P$ is integral. Then, the corresponding induced geodesic ray is defined by the path of symplectic potentials $g_t = g_0 + sf, s>0$.
\end{thm}

{}

\section{Quantization in new toric polarizations on $\mathbb{CP}^1$}\label{chapter:new; polarizations; supp not 1}

Our goal is to generalize the results in the previous section for a larger class of symplectic potentials; namely, we will consider $\psi$ to be a function whose Hessian is given in terms of bump-functions with localized supports in $P$. In this section, we will focus on the special case of $S^2\cong\mathbb{CP}^1$ and the general case will be described in Section \ref{sec_higherdim}.  We will consider the convergence both of $L^1$-normalized polarized sections (as in \cite{bfmn}) and of half-form corrected polarized sections using the gCST as in \cite{kmn2}.

\subsection{The case of one bump function}

\subsubsection{New toric polarizations in $\mathbb{CP}^1$}
\label{subsec_onebump}

Consider $S^2$ with the usual $S^1$-action and moment map $\mu$ given by the height function, such that the moment polytope is 
	$P=[0,N]$, where $N\in \mathbb{N}$. In this section, we will fix $m\in P$ and consider $\psi$ to be a function on $P$ such that its second derivative is a bump function, with maximum at $x=m$ and even about $x=m$, with support $$S=supp\,\psi''=[m-\alpha,m+\alpha]\cap P,$$ for sufficiently small $\alpha>0$. When $m\in \check P$, with $P_1=\left[0,m-\alpha\right]$ and $P_2=\left[m+\alpha, N\right]$ one obtains the profiles described in Figures 1,2 and 3. Note that letting \begin{equation} \label{areaofpsi''}
    A= \int_{S}\psi''dx,
\end{equation} 
we can take $\psi'(m-\alpha)=0$  and $\psi'(m+\alpha)=A$, in this case.
	\begin{figure}[h]
		\centering
		\begin{minipage}{0.45\textwidth}
			\includegraphics[width=\textwidth]{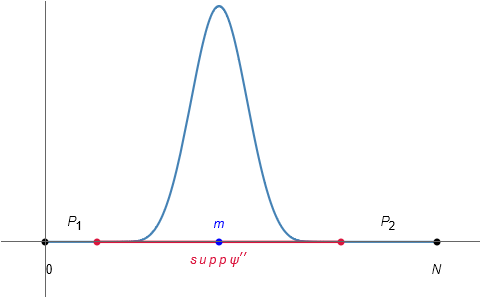}
			\caption{$\psi''$}
		\end{minipage}\hfill
		\begin{minipage}{0.45\textwidth}
			\includegraphics[width=\textwidth]{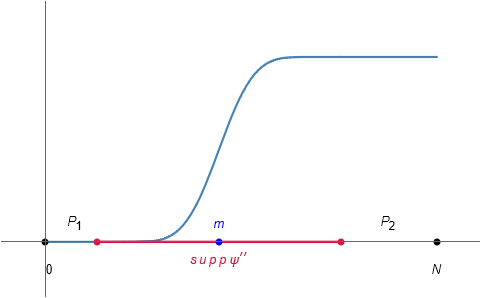}
			\caption{$\psi'$}
		\end{minipage}\par
		\vskip\floatsep% normal separation between figures
		\includegraphics[width=0.45\textwidth]{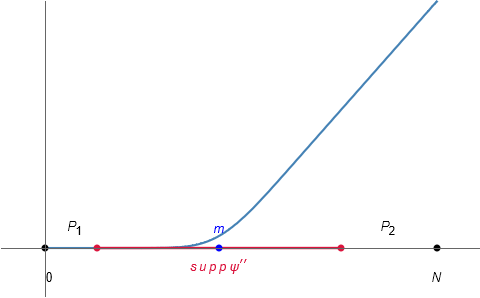}
		\caption{$\psi$}
	\end{figure}\\

 We will consider the family polarizations defined by the symplectic potentials $g_p+s\psi, \,\,s\geq 0$. 

{}

\begin{lemma}\label{openorbitlemma}
As $s\to \infty$, $\mathcal{P}_s$ converges pointwise in the Lagrangian Grassmannian to a limit polarization $\mathcal{P}_\infty$.
	On $$\breve{M}:=\mu^{-1}(\check S),$$ 
 where $\check S$ is the interior of $S$, 
 one has that \(\mathcal{P}_\infty=\mathcal{P}_\mathbb{R}.\) On the complement of this set relative to $\check M$ the polarization does not change with $s$, that is, 
 $$\mathcal{P}_s = \mathcal{P}_0,\,\, \text{on} \,\, \check M\setminus \breve M,\, s\geq 0.$$
\end{lemma}
\begin{proof}
		On $\breve{M}$, the proof follows exactly like in Lemma 3.2 in \cite{bfmn}. On $\check M\setminus \breve M,$ the polarization is given in terms of $\frac{\partial}{\partial w_s}$ which is independent of $s$ since $\psi''$ vanishes there.
\end{proof}

Including the points in the boundary of $P$, one obtains

\begin{prop}\label{anychartconvergence}One has the following properties of the limit polarization $\mathcal{P}_\infty$:
	\begin{enumerate}
		\item If $S\cap \partial P=\emptyset$ then, 
  $\mathcal{P}_\infty=\mathcal{P}_0$ on $M\setminus \breve M.$
		\item If $S\cap \partial P\neq\emptyset$, then on 
  $\mu^{-1}(S\cap \partial P)$ one has $$\mathcal{P}_\infty=
  \mathcal{P}_\mathbb{R}\oplus \langle\frac{\partial}{\partial w_j}\, \vert\, w_j=0\rangle_\mathbb{C}$$.
	\end{enumerate}
\end{prop}
\begin{proof}
	This follows exactly like Theorem 3.4 in \cite{bfmn}. 
 \end{proof}

{}
 
As a consequence, we arrive at an analog of Theorem \ref{funções convergem}.

\begin{thm}
\label{cinftyconvergence} 
On $\mu^{-1}(S),$
	\[C^\infty\left(\mathcal{P}_\infty\right)= C^\infty(\mathcal{P}_\mathbb{R}).\]
\end{thm}

\begin{proof}
	We have two cases, either if $S\cap \partial P=\emptyset$ or if $S\cap \partial P\neq\emptyset$. The first case follows from Proposition \ref{anychartconvergence}. In the other case, the proof follows exactly like in Theorem 1.2 in \cite{bfmn}.
\end{proof}

In the remainder of Section \ref{chapter:new; polarizations; supp not 1} we will take 
$m\in \check P$ for simplicity of exposition. The results generalize straightforwardly to the case when $m\in \partial P$.

\subsubsection{Convergence of $L^1$ normalized monomial sections}
In this section, as above, we take $m\in \check P$ and $\psi$ so that  $\psi''$ is a bump function even with respect to reflections around $m$. We also set  $\psi(m-\alpha)=\psi'(m-\alpha)=0$ so that 
$\psi'(x)=A$ for $x\geq m+\alpha$.

\begin{lemma}
\label{onP2}
	 For $x\geq m+\alpha$, 
 $\psi(x)=A(x-m)$.
\end{lemma}

\begin{proof}
	Let \[c:=\int_{m-\alpha}^{m+\alpha}\psi'(x)dx=\psi(m+\alpha).\]
	This result follows from the fact that 
	\[\left((x-m)\psi'\right)'=(x-m)\psi''+\psi',\]
	and so
	\begin{align*}
		\int_{m-\alpha}^{m+\alpha}\left((x-m)\psi'(x)\right)'dx&=\int_{m-\alpha}^{m+\alpha}[(x-m)\psi''(x)+\psi'(x)]dx &&\iff\\
		\alpha A&=c+\int_{m-\alpha}^{m+\alpha}(x-m)\psi''(x)dx &&\iff\\
		\alpha A&=c+\int_{-\alpha}^{\alpha}y\psi''(y+m)dy &&\iff\\
		\alpha A&=c = \psi(m+\alpha). 
	\end{align*}
	The fourth equivalence sign comes from the fact that the integrand is odd.
	Thus, for all $x\geq m+\alpha,\,\,\psi(x)=A(x-m)$.
\end{proof}
For $n\in P$, consider now the following function
\[f_n(x)=(x-n)\frac{\partial\psi}{\partial x}-\psi(x).\]

\begin{thm}
\label{fntheorem}
	Let $\psi$ be as above. Then
	\begin{enumerate}
		\item[(a)] For $n\in\check S,$ in the sense of distributions,
		\[\frac{e^{-sf_n}}{||e^{-sf_n}||_1}\xrightarrow{s\to\infty}\delta(x-n).\]
		\item[(b)] For $n\in P_j, j=1,2$, 
		\[\frac{e^{-sf_n}}{||e^{-sf_n}||_1}\xrightarrow{s\to\infty}\frac{1}{Vol(P_j)}\chi_{P_j }.\]
	\end{enumerate}
\end{thm}

{}

\begin{proof}
    We have,  \[f'_n(x)=(x-n)\psi''(x)\] and it follows that 
	\begin{align*}
		f_n(x)&=f_n(n)+\int_0^1\frac{d}{dt}f_n(n+t(x-n))dt\\
		&=-\psi(n)+\int_0^1t(x-n)^2\psi''(n+t(x-n))dt\\
		&\geq -\psi(n),
	\end{align*}
	where the last observation comes from the fact that the integral is always greater than or equal zero. So, the absolute minimum value of $f_n$ is $f_n(n)= -\psi(n)$. In fact, it also follows that if $n\in \check S$ then $x=n$ is an isolated minimum. On the other hand, $f_n=0$ on $P_1$ and $f_n(x)=(m-n)A$ for $x\in P_2$ which is the maximum value attained by $f_n$.
 Let $$\Delta_n(x)= \psi(n)+f_n(x)\geq 0, \, x\in P.$$
 We have
 $$
 \frac{e^{-sf_n}}{\vert\vert e^{-sf_n}\vert\vert_1} = \frac{e^{-s\Delta_n}}{\vert\vert e^{-s\Delta_n}\vert\vert_1}.
 $$
 
 Let $n\in \check S$. Then, for sufficiently small $\varepsilon >0$,
 $$
 \vert\vert e^{-s\Delta_n}\vert\vert_1 \geq \int_{B_\varepsilon(n)} e^{-s\Delta_n(x)}dx \geq Vol(B_\varepsilon (n)) e^{-\frac{s}{2}\varepsilon^2 \psi''(m)}.
 $$
Thus,
$$
\frac{e^{-s\Delta_n(x)}}{\vert\vert e^{-s\Delta_n}\vert\vert_1} \leq e^{-s\Delta_n(x)} e^{\frac{s}{2}\varepsilon^2 \psi''(m)} Vol (B_\varepsilon(n))^{-1}.
$$
If $x\neq n$, since $\Delta_n(x)>0$, we obtain, choosing sufficiently small $\varepsilon$, that 
$$
\frac{e^{-s\Delta_n(x)}}{\vert\vert e^{-s\Delta_n}\vert\vert_1} \to 0.
$$
This proves (a). On the other hand, if $n\in P_1$, then  $\Delta_n(x)=0$ for $x\in P_1$ and $\Delta_n(x)>0$ for $x\in \check S \cup P_2$. By the dominated convergence theorem, as $s\to\infty$ we obtain 
$\vert\vert e^{-s\Delta_n}\vert\vert_1\to Vol(P_1)$ and (b) follows. If now $x\in P_2$, we have that $\Delta_n(x)>0$ for $x\in P_1 \cup \check S$ and $\Delta_n(x)=0$ for $x\in P_2$. By the dominated convergence theorem, as $s\to\infty$, $\vert\vert e^{-s\Delta_n}\vert\vert_1\to Vol(P_2)$ and also (b) follows.

\end{proof}

{}

Thus, we arrive at the analog of Theorem \ref{important3}.
\begin{thm}\label{iotatheorem}
	For $n\in P\cap \mathbb{Z},$ consider the family of $L^1$-normalized $J_s$-holomorphic sections
	\[\mathbb{R}^+\ni s\mapsto \xi_s^n:=\frac{\sigma^n_s}{||\sigma^n_s||_1}\in C^\infty(L_\omega)\xhookrightarrow{\iota}(C^\infty_c(L^{-1}_\omega|_U))^*.\]
	Then,
		\begin{enumerate}
		\item[(a)] For $n\in\check S,$ as $s\to\infty,$ $\iota(\xi^n_s)$ converges to $\delta^n $ in $(C^\infty_c(L^{-1}_\omega|_U))^*$.
  
		\item[(b)] For $n\in P_j, j=1,2$, as $s\to\infty,$ $\iota(\xi^n_s)$ converges to $$\frac{1}{||\sigma^n_0||_{L^1(P_j)}} \int_{P_j}  \sigma^n_0\hat\tau(\cdot,-n)dx $$ in $(C^\infty_c(L^{-1}_\omega|_U))^*$.
		\end{enumerate}
\end{thm}
\begin{proof}
	The proof of this result follows \cite{bfmn}. In short, we can consider a partition of unity $\{\rho_v\}$ subordinated to the covering by vertex charts $\{$\u{P}$_v\}$, and therefore we only have to check the result in each chart. Thus choosing a test section $\tau\in C^\infty(L^{-1}_\omega)$, we may define
	\[h^s_n(x)=s(x-n)\psi'-g_s=h^0_n(x)-sf_n(x).\]
	Hence, following the computations in \cite{bfmn},
	\[(\iota(\xi^n_s))(\tau)=\frac{1}{||\sigma^n_s||_1}\int_Pe^{-h^s_n(x)}\hat\tau(x,-n)dx.\]
	Also
	\[||\sigma^n_s||_1=\int_{\check M}e^{-h^s_n\circ\mu_P}\omega^n=(2\pi)^n\int_Pe^{-h^s_n}dx.\]
	We have
	\[\frac{||e^{-h^0_n-sf_n}||_1}{||e^{-sf_n}||_1}=\int_P\frac{e^{-sf_n}}{||e^{-sf_n}||_1}e^{-h^0_n}dx.\]
	Now, using Theorem \ref{fntheorem} we obtain
	\begin{itemize}
		\item For $n\in\check S,$ \[\frac{||e^{-h^0_n-sf_n}||_1}{||e^{-sf_n}||_1}\xrightarrow{s\to\infty}e^{-h^0_n(n)};\]
		\item For $n\in P_j, j=1,2,$ using the dominated convergence theorem, we have \[\frac{||e^{-h^0_n-sf_n}||_1}{||e^{-sf_n}||_1}\xrightarrow{s\to\infty}\frac{1}{Vol(P_j)}\int_{P_j}e^{-h^0_n(x)}dx;\]
	\end{itemize}
	Which then implies that:
	\begin{itemize}
		\item for $n\in\check S,$
		\begin{align*}
			\iota(\xi_s^n)(\tau)&=\int_P\frac{e^{-h^0_n-sf_n}}{||e^{-h^0_n-sf_n}||_1}\hat\tau(\cdot,-n)dx\\
			&=\int_P\frac{e^{-h^0_n-sf_n}}{||e^{-h^0_n-sf_n}||_1}\frac{||e^{-sf_n}||_1}{||e^{-sf_n}||_1}\hat\tau(\cdot,-n)dx\\
			&\xrightarrow{s\to\infty} \int_P e^{-h^0_n}e^{h^0_n}\delta^n\hat\tau(\cdot,-n)dx\\
			&= \delta^n(\tau);\\
		\end{align*}
		\item For $n\in P_j, j=1,2,$ using the dominated convergence theorem:
		\begin{align*}
			\iota(\xi_s^n)(\tau)&=\int_P\frac{e^{-h^0_n+sf_n}}{||e^{-h^0_n+sf_n}||_1}\hat\tau(\cdot,-n)dx\\
			&=\int_P\frac{e^{-sf_n}}{||e^{-sf_n}||_1}\frac{||e^{-sf_n}||_1}{||e^{-h^0_n+sf_n}||_1}e^{-h^0_n}\hat\tau(\cdot,-n)dx\\
			&\xrightarrow{s\to\infty} \int_P\frac{1}{Vol(P_j)}\chi_{P_j } \frac{Vol(P_j)}{||\sigma^n_0||_1 }e^{-h^0_n}\hat\tau(\cdot,-n)dx\\
			&= \frac{1}{||\sigma^n_0||_{L^1(P_j)}} \int_{P_j}  \sigma^n_0\hat\tau(\cdot,-n)dx.
		\end{align*}
	\end{itemize}
\end{proof}

{}

\begin{rmk}\label{bigremark}
	These results show that these Mabuchi geodesics of toric K\"ahler structures allow for a decompostion of the phase space associated to the decomposition of the polytope,
$$
P = P_1 \cup S\cup P_2.
$$
	Suppose, for instance, that there are no integral points on the support of $\psi$. In this case, the monomial sections $\sigma^n$ converge to their normalized restriction to the corresponding part of the polytope, i.e. $P_1$ if $n<m$ or $P_2$ if $n>m$. Moreover, as in the limit these sections only have support on the corresponding $P_i$, we obtain that, at infinite geodesic time along the Mabuchi geodesic, the Hilbert space for the quantization on the whole $S^2$, decomposes into the Hilbert spaces for quantization of the two subsets $\mu^{-1}(P_1)$ and $\mu^{-1}(P_2)$. Note that, from the metric point of view,  these two regions in phase space become separated by infinitely long lines.
	
	If the support does contain at least one integral point, we will have three regions in phase space  that support polarized sections in the limit of infinite geodesic time. In particular, the sections supported in $\text{supp}\,\,\psi''$ converge to distributional sections. Again, in the limit, the quantization of $S^2$ decomposes into the quantizations of the three subregions which will be separated from each other, metrically, by infinitely long lines.
	
	This result is quite interesting as, in general, there is no way of ``decomposing" a phase space into subsets, in some
	geometrically natural way, in such a way that the quantization of the symplectic manifold
	also ``decomposes" as a sum of the quantizations of those subsets. 
\end{rmk}
\begin{rmk}\label{bigpolytoperemark}
	Notice that all the above results are still valid on the plane, where when the moment polytope is of the form $[0,+\infty)$.
\end{rmk}

\subsubsection{Generalized coherent state transfom and half-form quantization}
\label{subsubseccst}

In this section, we study the behaviour of half-form corrected holomorphic quantization, studied in the toric case in \cite{kmn},
along the Mabuchi ray of toric K\"ahler structures describes in Sections \ref{subsec_onebump}. We will follow \cite{kmn2} and consider generalized coherent state transforms as recalled in Proposition \ref{cst1} and Theorem \ref{compleximportanttheorem}.

Following \cite{kmn}, we consider the corrected polytope \[P=\left[-1/2,N+1/2\right],\] and as such, $P_1$ and $P_2$ are also corrected by shifts of $1/2$.  Recall generaized coherent state transform defined by the operator $$A^\psi_{s}:\mathcal{H}_{0}\rightarrow\mathcal{H}_{s},$$ for $s\geq 0$, defined by
	\[ A^\psi_{s}:=\left(e^{s\hat\psi}\otimes e^{is\mathcal{L}_\psi}\right)\circ e^{-s\hat\psi_\mathbb{R}}. \]

\begin{thm}\label{novocomplexa}
	 Then,
	\begin{enumerate}
		\item If $n\in \check S\cap \mathbb{Z}$, \[ \lim_{s\to \infty} A^\psi_{s}\left( {\sigma_0^n \otimes \sqrt{dZ_0}}{}\right)={(2\pi)^{\frac{n}{2}}e^{g_P(n)}}{}\delta^n\otimes\sqrt{dX},\] 
		\item If $n\in P_i\cap \mathbb{Z}, i=1,2$, \[ \lim_{s\to \infty}A^\psi_{s}\left( {\sigma_0^n\otimes d \sqrt{dZ_0}}{}\right)=\chi_{P_i}\sigma^n_0\otimes\sqrt{dZ_0}.\] 
	\end{enumerate}
	
\end{thm}
\begin{proof}
	The first case follows from the proof of  Theorem \ref{compleximportanttheorem} in \cite{kmn}. For the case, $n\in P_i\cap \mathbb{Z}$. from Proposition \ref{cst1} and again from 
	the proof of the Theorem \ref{compleximportanttheorem},  we have that \[A^\psi_{s}\left( {\sigma_0^n\otimes\sqrt{dZ_0}}{}\right)={e^{-s\psi(n)}\sigma_s^n\otimes \sqrt{dZ_s}}{}.\]
	Assume now that $n\in P_1$, so that  $\psi(n)=0$. From the proof of Theorem \ref{iotatheorem} and by the dominated convergence theorem, 
	\begin{align*}
		&\int_P e^{-sf_n(x)}e^{-((x-n)g_P'(x)-g_P(x))}dx=\int_{P_1}e^{-((x-n)g_P'(x)-g_P(x))}dx  + \\ &+\int_{P\backslash{P_1}}e^{-sf_n(x)}e^{-((x-n)g_P'(x)-g_P(x))}dx
		\xrightarrow{s\to\infty} \int_{P_1}e^{-((x-n)g_P'(x)-g_P(x))}dx
	\end{align*}
	which implies our result. Assume now that $n\in P_2$. Then $\psi(n)=(n-m)$. For any $\varepsilon>0$, we obtain that 
	\begin{align*}
		&\int_{[-1/2,m+\alpha-\varepsilon]}e^{-sA(n-m)}e^{-sf_n(x)}e^{-((x-n)g_0'(x)-g_0(x))}dx\\&\leq
		Be^{-sA(n-m)}e^{-sf_n(m+\alpha-\varepsilon)}
		\xrightarrow{s\to\infty}0,
	\end{align*}
	where $B=(m+\alpha-\varepsilon+\frac{1}{2})e^{-((\xi-n)g_P'(\xi)-g_P(\xi))}$, where $\xi$ is the maximum of the exponential factor in this interval. The convergence follows from the fact that $A(n-m)=-f_n(m+\alpha)$ and that $f_n$ is a decreasing non-positive function. Moreover, notice that on $P_2$, we have that 
	\begin{align*}
		&\int_{P_2}e^{-sA(n-m)}e^{-sf_n(x)}e^{-((x-n)g_P'(x)-g_P(x))}dx\\&=\int_{P_2}e^{-sA(n-m)}e^{-sA(m-n)}e^{-((x-n)g_P'(x)-g_P(x))}dx
		=\int_{P_2}e^{-((x-n)g_P'(x)-g_P(x))}dx
	\end{align*}
	which, together with the proof of Theorem \ref{iotatheorem}, proves our claim.
\end{proof}

\subsection{Generalization to more bump functions}
\label{subsec_morebumps}

Let us now further generalize the results of the previous sections to the case when the second derivative of $\psi$ consists of more than one bump function. 
Consider, then,  the case when $\psi''$ consists in $K$ bump functions with disjoint supports, $S_j\subset P =[0,N], j=1, \dots, K,$ each with integral $A_j$. Let $S_j:=[m_j-\alpha_j,m_j+\alpha_j]$, $m_j\in P$, $\alpha_j>0$, $j=1,...,K$ be the supports of the $K$ bump functions, where we take $m_1 < m_2< \cdots < m_K$. Then, generalizing the previous notation, let \[P_1=[0,m_1-\alpha_1],\,\,P_{K+1}=[m_K+\alpha_K,N]\]\[ P_j=[m_{j-1}+\alpha_{j-1},m_j-\alpha_j], \,\, j=2,...,K.\] 
 Taking
 $$
\sum_{k=1}^{j-1} \int_{m_k-\alpha_k}^{m_k+\alpha_k} ((x-m_k)\psi'(x))'dx,
 $$
 we obtain
\[\psi(x)=(x-m_{j-1})\sum_{k=1}^{j-1} A_k + \sum_{k=1}^{j-2} A_k(m_{j-1}-m_k) ,\,\, \forall x\in P_j,\,\, j\neq 1.\]

As before, let
	\[f_n(x)=(x-n)\frac{\partial\psi}{\partial x}-\psi(x),\]
where $n\in P$. It is straightforward to verify that $f_n$ is constant in each $P_l$ with 
\begin{equation}\label{valuefn}f_n(x) = \sum_{j=1}^{l-1}(m_j-n)A_j, \, x\in P_l.\end{equation}

We obtain the following generalization of the Theorem \ref{fntheorem}:

{}

\begin{thm}\label{fntheoremgeneral}
	Let $\psi$ be as described above and let $S=\cup_{j=1}^K S_j$.  Then
	\begin{enumerate}
		\item[(a)] For $n\in \check S$, in the sense of distributions,
		\[\frac{e^{-sf_n}}{||e^{-sf_n}||_1}\xrightarrow{s\to\infty}\delta(x-n).\]
		\item[(b)] For $n\in P_j$,   
		\[\frac{e^{-sf_n}}{||e^{-sf_n}||_1}\xrightarrow{s\to\infty}\frac{1}{\text{Vol}(P_j)}\chi_{P_j },\,\, j=1,\dots, K+1.\]
	\end{enumerate}
\end{thm}

{}
\begin{proof}
    Consider first the case when $n\in \check S$. As in Theorem \ref{fntheorem}, we have that $x=n$ is the absolute minimum value of $f_n$, with $f_n(n)= -\psi(n)$. Again, 
    $\Delta_n(x)=\psi(n)+f_n(x)\geq 0, x\in P$. Let $n\in \check S_l \subset \check S.$ Then, the estimate for $\vert\vert e^{-s\Delta_n}\vert\vert_1$ using a ball of radius $\varepsilon$ around $x=n$, as in the proof of Theorem \ref{fntheorem}, is still valid and since $\Delta_n(x)>0$ for $x\neq n$ we obtain (a). On the other hand, if $n\in P_l,$ then
    $\Delta_n(x)>0$ if $x\notin P_l$ while $\Delta_n(x)=0$ for $x\in P_l$. Thus, from the dominated convergence theorem, as $s\to \infty,$
    $$\vert\vert e^{-s\Delta_n}\vert\vert_1\to Vol (P_l)$$ and (b) follows.
\end{proof}

Furthermore, the analogs of Theorems \ref{iotatheorem} and \ref{novocomplexa}, are still valid and their proofs are essentially the same in the case when $\psi''$ consists of several bump functions.
Also, Remarks \ref{bigremark} and \ref{bigpolytoperemark} are still valid in this case, with the appropriate obvious adaptations.

\section{Higher dimensional toric manifolds}
\label{sec_higherdim}

{}
We will now generalize the previous results to higher dimensional symplectic toric manifolds. Let $M$ be a  toric K\"aler manifold with moment polytope $P$ and moment map $\mu \, : \, M \rightarrow P$, and consider 
 a polyhedral decomposition of $P$ associated to a piecewise linear rational convex function $f$ on $P$ defining
the polytope
$$
Q = \left\{(x, x_{n+1}) \in P \times \mathbb{R} \, : \, 
0 \leq x_{n+1}  \leq  K- f(x) \, \right\} \,
, $$
where $K \geq {\rm max} \left\{f(x) \, , x \in P \right\} $. We will assume  that $Q$ is integral so that it 
defines an equivariant line bundle, $\mathcal{L} \rightarrow \mathcal{M}$, and the corresponding toric test configuration \cite{donaldson02, chentang08}. 
The central fiber of the family
consisting of the union of toric
manifolds associated with the 
polyhedral decomposition of $P$
and with the ceiling of $Q$,
corresponds to the limit 
of the Mabuchi geodesic with initial velocity given by the Legendre transform of $f$ (in the complex picture). 

As we will show below, in the symplectic picture, 
one obtains a family of mixed polarizations on $M$,
$\mathcal{P}_{\infty, \epsilon}$,
which are obtained at infinite geodesic time along Mabuchi geodesics generated by smoothings, 
$f_\epsilon$,
of $f$. 
The polarization $\mathcal{P}_\infty$, obtained in the limit $\epsilon\to 0,$ 
is then the symplectic picture analogue of the 
central fiber.

The quantization of $M$ with respect to these polarizations has properties analogous to the ones described in the previous Section for $\mathbb{C}\mathbb{P}^1$ in Theorems \ref{iotatheorem} and 
\ref{novocomplexa}. In particular, we  provide  analogues of Proposition \ref{anychartconvergence}, Theorem \ref{cinftyconvergence}, and Theorem \ref{fntheorem}, as the other results follow  from these. 

Let $M$ be a $2n$-dimensional symplectic toric manifold with Delzant moment polytope $P$ and 
let $f:P\to \mathbb{R}$ be a rational piecewise linear convex function associated to test configuration, 
as recalled in Section \ref{sec_testconfig}, 
$$
f(x) = \mathrm{max} \left\{ a_1(x), \dots, a_r(x) \right\},
$$
where  the $a_j, j=1, \dots, r$, are affine linear functions with rational coefficients. Let $W\subset P$ be the locus of non-differentiability of $f$, so that the collection of connected components ${\rm int}(P_j)$ of $P\setminus W$ 
gives a decomposition of $P$ into closed sub-polytopes $P_j, j=1, \dots, q$, such that 
$$
P = \bigcup_{j=1}^q P_j.
$$

Note that each facet of $P_j$ which is not contained in a facet of $P$  is shared with another $P_k$ and corresponds to 
the equality of some pair of affine linear functions, $a_l=a_{l'}.$ Likewise, higher codimension faces of $P_j$ not in $\partial P$ will correspond to loci where more than two affine linear functions 
$a_j,  j=1, \dots, r$, become equal. We will refer collectively to the faces of the sub-polytopes $P_j$ as ``faces of $W$".

In the following, we will assume that the sub-polytopes $P_j, j=1, \dots, q$, are themselves Delzant polytopes and below we will consider ``nice" convex smoothings of $f$ satisfying appropriate  conditions. 

Let $\mathcal{F}_j, j=1, \dots, n$, be the set of closed codimension-$j$ faces of $W$.
For a face $F\in \mathcal{F}_j$ with primitive normals $\nu_k^F, k=1, \dots, j$, let 
$(x_F, x_F^\perp)\in \mathbb{R}^{n-j}\times \mathbb{R}^{j}$ denote moment coordinates obtained from $(x_1, \dots, x_n)$ by an appropriate $SL_n(\mathbb{Z})$ transformation, such that $F$ is contained in an hyperplane of the form $x_F^\perp=c_F\in \mathbb{R}^j.$

For sufficiently small $\varepsilon >0$, consider an $\varepsilon$-thickening $W_\varepsilon\supset W$ of $W$, obtained by 
\begin{equation}
\label{e.we}
W_\varepsilon = P \bigcap \left( \bigcup_{j=1}^{n} \bigcup_{F\in \mathcal{F}_j} \check F \times ]c_F^1-\varepsilon, c_F^1+\varepsilon[ \times \cdots \times ]c_F^j-\varepsilon, c_F^j+\varepsilon[ \right) ,
    \end{equation}
where $\check F$ denotes the relative interior of $F$. 
{}

\begin{dfn}Let $\epsilon>0$ be sufficiently small and $I_\epsilon = (0, \epsilon)$. We 
say that $\left\{\psi_\varepsilon:P\to \mathbb{R}, \varepsilon\in I_\epsilon \right\}$ is a ``nice" family of smoothings of $f$ if:
\begin{enumerate}
    \item[a)] for each $\varepsilon$, $\psi_\varepsilon$ is smooth and convex;
    \item[b)] for each $p\in P$, $\psi_\varepsilon(p)$ is smooth in $\varepsilon$;
    \item[c)] for each $\varepsilon>0$, $\psi_\varepsilon =f$ on $P\setminus W_\varepsilon$;
    \item[d)] For each $p\in \check F\subset \mathcal{F}_j$, the rank of $\mathrm{Hess }\,\psi_\varepsilon (p)$ is $\geq j$ and its restriction to the subspace generated by the $x_F^\perp$-directions is positive definite;
    \item[e)] For each $p\in \check F\subset \mathcal{F}_j$, there exists $\varepsilon_p\in I_\epsilon$ such that for all $\tilde \varepsilon < \varepsilon_p$, 
    $\mathrm{Hess }\,\psi_{\tilde \varepsilon} (p)$ has rank $j$. 
\end{enumerate}    
\end{dfn}

Note that the support of $\mathrm{Hess} \,\psi_\varepsilon$ is the closure $\overline W_\varepsilon$ of $W_\varepsilon$.

\begin{thm}
    There exist nice families of smoothings of $f$.
\end{thm}

{}

\begin{proof}
    From Theorem 2.1 and Condition 3' in Section 3 in \cite{Ghomi02}, we can find smoothings of $f$ which are strictly convex on an open neighborhood $V\supset W$ and which coincide with $f$ on $P\setminus V$. (See also Remark \ref{strict} about the properties of quantization along Mabuchi geodesics generated by such smoothings.) Such smoothings are obtained by pasting $f$ together with 
    the convolution of $f$ with convolution kernels in an open neighbourhood of $W$.

    Let us first recall the construction in \cite{Ghomi02}. Let 
    $\theta_\epsilon \in C^\infty (\mathbb{R}^n)$, with $\theta_\epsilon (y)\geq 0, \forall y\in \mathbb{R}^n$, have support on a ball of radius $\epsilon>0$ around the origin, be even under the reflection $y\to -y$, and satisfy
    $\int_{\mathbb{R}^n}\theta_\epsilon (y)dy =1.$ 
  Convolution with $\theta_\epsilon$ gives,
  $$
  \tilde f_\epsilon (x) = \int_{\mathbb{R}^n} f(x-y)\theta_\epsilon(y) dy = 
  \int_{\mathbb{R}^n} f(y) \theta_\epsilon (x-y) dy.
  $$
  Then, $f_\epsilon$ is smooth and for any compact set $K$ where $f$ is of class $C^r$, 
  $$
  \lim_{\epsilon\to 0}
  \vert\vert f- \tilde f_\epsilon\vert\vert_{C^r(K)} =0.
  $$
  Consider tickenings $W_{\epsilon_j}, j=1, 2$ of $W$ with $\epsilon_2<\epsilon_1$ and $\overline W_{\epsilon_2}\subset W_{\epsilon_1}.$ 
  Let $\phi \in C^\infty (\mathbb{R}^n)$ be a bump function with $\phi =1$ on $P\setminus W_{\epsilon_1}$ and support on $P\setminus W_{\epsilon_2}.$
  Set $f_\epsilon = (1-\phi) \tilde f_\epsilon + \phi f.$ Then $f_\epsilon$
 is smooth, and for sufficiently small $\epsilon$ convexity of $f$ implies convexity of $f_\epsilon$ and $f_\epsilon=f$ along $P\setminus W_{\epsilon_1}.$

 In the present case, in order to achieve the conditions on the rank of the Hessian of the smoothing of the piecewise linear function $f$, we will use convolution kernels only along the directions normal to faces of $W$. Let $F\in \mathcal{F}_j, j>0,$ be a face of $W$, with parallel and normal coordinates $x_F, x_F^\perp.$ Consider a convolution kernel $\theta_\epsilon \in C^\infty (\mathbb{R}^{j})$ as above, even under reflections $y \to -y,$ but which we will take only along the ``non-smooth" directions of $f$ at $F$ given by $x_F^\perp$. Let 
 $$
 \tilde f_\epsilon^F (x_F, x_F^\perp) = \int_{\mathbb{R}^j} f(x_F, x_F^\perp-y) \theta_\epsilon (y) dy,
 $$
 which is smooth in an open neighbourhood $A$ such that of $A\cap F= \check F$ and such that $A$ does not overlap any other face of $W$. Note that near $\check F$, $f$ is linear in $x_F$ so that $\tilde f_\epsilon^F$ is also linear in $x_F$ near $\check F$.  
 Moreover,  note that, near $F$, strict convexity of $f$ relative to the directions $x_F^\perp$ implies that the Hessian of $\tilde f_\epsilon$ has exactly rank $j$ and is positive definite in the normal directions $x_F^\perp$ in a neighbourhood of $\check F$. (See the argument in the proof of Theorem 2.1 of \cite{Ghomi02}.)
 Let $V\subset U$ be sufficiently small open sets such that $U\cap F$ and $V\cap F$ are relatively open sets in $\check F$.
 Let $\phi_F \in C^\infty(\mathbb{R}^n)$ be a bump function which is equal to 1 on the complement of $U$ and with support in the complement of $V$. Then,
 $$
 f_\epsilon = (1-\phi_F) \tilde f_\epsilon^F + \phi_F f
 $$
 is smooth, $f_\epsilon = f$ on the complement of $U$ and $f_\epsilon = f_\epsilon^F$ in $V$. For $p\in U\setminus V$, for sufficiently small  
 $\epsilon$, the convolution of $\theta_\epsilon$ with $f$ 
 will be over a region where $f$ is linear (in both $X_F$ and $x_F^\perp$) so that $f=\tilde f^F_\epsilon$ in an open neighbourhood of $p$ for sufficiently small $\varepsilon$. Thus, in the intermediate region $U\setminus V$ $f_\epsilon$ is convex (and linear) for sufficiently small $\epsilon$. 
 By changing $f$ along neighbourhoods of relatively open sets in the interior of each face $F$ of $W$ we obtain the desired smoothings and the nice families follow by changing the size of the neighbourhoods.
 
 \end{proof}

{}

Given a nice family  of smoothings of $f$, $\psi_\epsilon$, 
let $\mathcal{P}_{s,\varepsilon}$, for $s>0$, be the toric K\"ahler polarization corresponding to the 
symplectic potential $g_s = g_P + s\psi_\varepsilon$. Along the corresponding Mabuchi geodesic we obtain at infinite geodesic time, $s\to \infty$, new mixed toric polarizations $\mathcal{P}_{\infty,\varepsilon}$ 
which are 
higher dimensional analogs of the mixed polarizations in $S^2$ described in Section \ref{chapter:new; polarizations; supp not 1}.

\begin{rmk}
As mentioned above, from its construction, the mixed polarization
$$
\mathcal{P}_{\infty} = \lim_{\varepsilon \to 0} \, \mathcal{P}_{\infty,\varepsilon}   \, ,
$$
as described below, 
corresponds to the symplectic picture version of the central fiber 
of the test family. 
From the metric point of view, as $\varepsilon\to 0$, we obtain that the metric remains unchanged along the Mabuchi flow on the subset $P\setminus W_\varepsilon$, while on codimension-$1$ facets in $W_\varepsilon$ the normal directions collapse metrically to points. This limit, which is taken in the symplectic picture, is a different limit from the one considered in \cite{chentang08, sz2012}, where there is a global (singular) complex structure on the central fiber and where one gets smooth $n$-dimensional toric varieties connected by lower dimension sub-varieties.    
\end{rmk}

\begin{figure}[h]
	          \centering
			\includegraphics[scale=0.5]{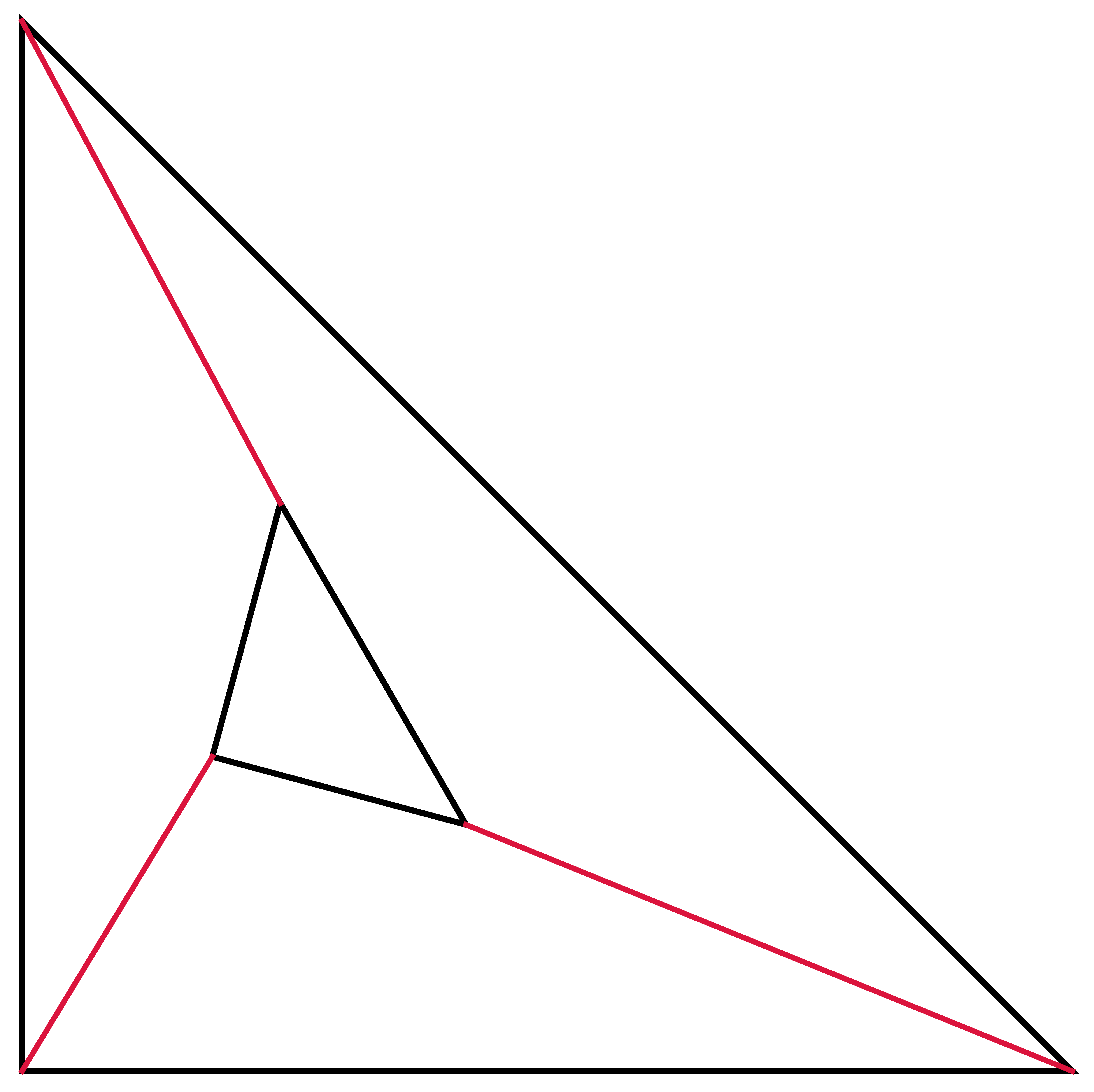}
			\caption{Example for $\mathbb{CP}^2$}
   %\label{figure1:CP2}
\end{figure}

\begin{dfn}
Let $\mathcal{P}_{\infty,\varepsilon}$ be the polarization that on $\mu^{-1}(P \setminus W_\varepsilon)$ coincides with the initial K\"ahler polarization and on
points $p \in \mu^{-1} \left(W_\varepsilon \right)$ is defined as follows. Let $F$ be the face of minimal dimension, $q=n-j$ such that 
$$\mu(p) \in \check F \times ]c_F^1-\varepsilon, c_F^1+\varepsilon[ \times \cdots \times ]c_F^j-\varepsilon, c_F^j+\varepsilon[ \, .
$$
Then
\begin{equation}\nonumber
%\label{eePs}
	\mathcal{P}_{\infty, \varepsilon}(p)=\text{span}_{\mathbb{C}}\left\{ \frac{\partial}{\partial (\theta_F^\perp)_i},\,\,i=1,...,n-q, \, 
 \frac{\partial}{\partial {(w_F})_j},\,\,j=1,...,q, \right\}.
\end{equation}

{}
\end{dfn}

\begin{thm} One has, 
\begin{enumerate}
    \item[a)] As $s\to +\infty$, pointwise in $P$, 
$$
\lim_{s\to \infty} \mathcal{P}_{s,\varepsilon} = \mathcal{P}_{\infty,\varepsilon}
$$   
where on  $\mu^{-1}(P\setminus W_\varepsilon)$, $C^\infty(\mathcal{P}_{\infty,\varepsilon}) = C^\infty(\mathcal{P}_0)$.

\item[b)] As $\varepsilon\to 0$, 
$$
\mathcal{P}_{\infty,\varepsilon}\to \mathcal{P}_\infty,
$$
where along $\check F\in \mathcal{F}_j$ $\mathcal{P}_\infty$ has $j$ real directions spanned by $\frac{\partial}{\partial \theta_F^\perp}$ and $n-j$ complex directions spanned by 
$\frac{\partial}{\partial (w)_F}$, where $w$ denotes the holomorphic coordinates along the open dense orbit for the symplectic potential $g_0.$
\end{enumerate}
\end{thm} 

\begin{proof}
    This just follows directly from the proof of Theorem 1.2 in \cite{bfmn} and from \cite{W22}, \cite{P22}.
\end{proof}
{}

Let now, for $m \in P \cap {\mathbb Z}^n$,
$$
f_{m}^\varepsilon = (x-m)\cdot \nabla \psi_\varepsilon - \psi_\varepsilon 
$$
and let $\check P_j$ denote the interior of the sub-polytope $P_j$.

\begin{prop}\label{newthmone}
    We have, for sufficiently small $\varepsilon$,
    \begin{enumerate}
        \item[a)] if $m\in \check P_j, j=1, \dots, q$,
        $$
        \frac{e^{-sf_m^\varepsilon}}{||e^{-sf_m^\varepsilon}||_1}\xrightarrow{s\to\infty}\frac{1}{\text{Vol}(P_j-W_\varepsilon)}\chi_{P_j-W_\varepsilon }.
        $$

        \item[b)] if $m\in \check F \in \mathcal{F}_j, j=1, \dots, n$,  
        $$
        \frac{e^{-sf_m^\varepsilon}}{||e^{-sf_m^\varepsilon}||_1}\xrightarrow{s\to\infty} \delta^j(x_F^\perp - n_F^\perp),
        $$
        where $\delta^j$ denotes the Dirac delta distribution localized on the hyperplane $x_F^\perp = c_F = n_F^\perp$ and $n_F^\perp$ denotes the second component of $n$ with respect to  the coordinates $(x_F, x_F^\perp).$
    \end{enumerate}
\end{prop}

\begin{proof}
    As in the proof of Theorem \ref{fntheorem}, convexity of $\psi_\varepsilon$ implies that $x=m$ is a global minimum for $f_m^\varepsilon$ with value $-\psi_\varepsilon(m)$. Let $\Delta_m^\varepsilon= \psi_\varepsilon (m)+f_m^\varepsilon\geq 0.$ If $m\in \check P_j$ for some $j=1, \dots, q$, 
    then for sufficiently small $\varepsilon$, $n\in P_j- W_\varepsilon$ and  
    $\Delta_m^\varepsilon$ is constant and equal to zero on $P_j-W_\varepsilon$ while it is  positive on the complement of this set. Thus, $e^{-sf_m^\varepsilon}$ converges pointwise to $0$ on the complement of  $P_j-W_\varepsilon$ and by the dominated convergence theorem
    $\vert\vert e^{-sf_m^\varepsilon}\vert\vert_1 \to Vol (P_j-W_\varepsilon)$ as $s\to \infty.$
    This proves $a)$.

    If $m\in \check F\in \mathcal{F}_j$, then for sufficiently small $\varepsilon$,  
    $\mathrm{Hess}\, \psi_\varepsilon (m)$ has rank $j$ and is positive definite in the directions of $x_F^\perp.$ Thus, the analog of the estimate for 
    $\vert\vert e^{-sf_m^\varepsilon}\vert\vert_1$ in the proof on Theorem \ref{fntheorem} is still valid where we integrate on a small interval $B=(V\cap \check F_j) \times ]c_F^1-\varepsilon', c_F^1+\varepsilon'[ \times \cdots \times ]c_F^j-\varepsilon', c_F^j+\varepsilon'[$, where $V$ is a small open set containing $n$ and $\varepsilon'$ is sufficiently small,
    $$
    \vert\vert e^{-s \Delta_n^\varepsilon} \vert\vert_1 \geq Vol (B) e^{-\frac{s}{2}j\lambda (\varepsilon')^2},
    $$
    where $\lambda$ is the smallest nonzero eigenvalue of $\mathrm{Hess}\, \psi_\varepsilon (m)$. 
    On the other hand, $\Delta_m^\varepsilon$ is constant and equal to zero 
    along $V'\cap \check F_j$ for an open set $V'$ containing $m$, while it is positive on the complement of that set. The dominated convergence theorem then implies $b)$. 
    \end{proof}

{}

\begin{rmk}
    The distribution $\delta^j(x_F^\perp-m_F^\perp)$ can be defined as the distribution which acting on a continuous test function with compact support $\varphi$ gives
    $$
    \left(\delta^j(x_F^\perp-m_F^\perp) \cdot \varphi\right) (x_F, x_F^\perp) = \varphi (x_F, m_F^\perp).
    $$
    (See \cite{hor}.)
\end{rmk}

Let $\sigma^m_{s,\varepsilon}$ be the $\mathcal{P}_{s,\varepsilon}$-polarized monomial holomorphic sections 
and let $\sqrt{dZ_{s,\varepsilon}}$ denote the corresponding sections of the half-form bundle.
Theorem \ref{newthmone} then implies the analog of Theorem \ref{iotatheorem}, by a reasoning exactly analogous to the one in the proof of that theorem,  

\begin{thm}\label{newthmtwo} For sufficiently small $\varepsilon$, in the sense of distributional sections,
\begin{enumerate}
    \item[a)] If $m\in \check P_j\cap \mathbb{Z}^n$, for some $j=1, \dots, q$, then 
    $$
    \frac{\sigma^m_{s,\varepsilon}}{\vert\vert \sigma^m_{s,\varepsilon} \vert\vert_1} \xrightarrow{s\to\infty} \frac{\sigma^m_{0}\chi_{P_j-W_\varepsilon}}{\vert\vert \sigma^m_{0} \vert\vert_{L^1(P_j-W_\varepsilon)}};
    $$

    \item[b)] If $m\in \check F\cap \mathbb{Z}^n, F\in \mathcal{F}_j$, then
    $$
    \frac{\sigma^m_{s,\varepsilon}}{\vert\vert \sigma^m_{s,\varepsilon} \vert\vert_1} \xrightarrow{s\to\infty} \delta^j(x_F^\perp-m_F^\perp)\frac{\sigma^{m_F}_0}{\vert\vert \sigma^{m_F}_0\vert\vert_1},
    $$
    where $\sigma^{m_F}_0$ denotes the $\mathcal{P}_0$-polarized monomial section for $m=(m_F,0).$
    \end{enumerate}
\end{thm}

{}

By the same reasoning of the proof of Theorem \ref{novocomplexa} we also obtain the generalization

\begin{thm}\label{newtheoremthree}Let $A_s^{\psi_\varepsilon}$ be generalized coherent state transform the operator defined in Sections \ref{subsec_quanttoric} and \ref{subsubseccst}. Then, for sufficiently small $\varepsilon$,
\begin{enumerate}
    \item[a)] If $m\in \check P_j\cap \mathbb{Z}^n$, for some $j=1, \dots, q,$
    $$
    \lim_{s\to\infty} A_s^{\psi_\varepsilon} \left(\sigma_0^m \otimes \sqrt{dZ_0}\right) = \chi_{P_j-W_\varepsilon} \sigma_0^m\otimes \sqrt{dZ_0};
    $$

    \item[b)] If $m\in \check F\cap \mathbb{Z}^n, F\in \mathcal{F}_j$, then
    $$
    \lim_{s\to\infty} A_s^{\psi_\varepsilon} \left(\sigma_0^m \otimes \sqrt{dZ_0}\right) =
    \mathrm{const.} \, \delta^j(x_F^\perp-m_F^\perp) \sigma_0^{m_F} \otimes \sqrt{dX_F^\perp \otimes d(Z_F)_0},
    $$
    where $dX_F^=dx_F^1\wedge \cdots \wedge dx_F^{j}$ and $d(Z_F)_0^\perp = d(z_F)^1\wedge \cdots \wedge d(z_F)_{n-j}.$
\end{enumerate}
\end{thm}

\begin{rmk}
    Sections of the form present in Theorem \ref{newthmtwo} 
    and \ref{newtheoremthree}, which are plarized with respect to toric mixed polarizations, with both real and holomorphic directions, have been studied in \cite{P22} and 
    \cite{W22,CLW22,CLW23,CLW23-2}.
\end{rmk}

{}

\begin{rmk}\label{strict} Note that using directly the smoothings of convex piecewise linear functions described in \cite{Ghomi02} (see also \cite{Ghomi04}), one obtains Mabuchi geodesics which are close to the ones that we describe above but where one has strict convexity along the tickening of $W$. For these geodesics, in the limit of infinite geodesic time, one will obtain the original K\"ahler polarizarion on $P\setminus W$ while along $W$ one obtains the real toric polarization. Accordingly, monomial holomorphic sections corresponding to integral points in $W$ will converge to the corresponding Dirac delta distributions. It is interesting to find Mabuchi geodesics whose velocities are arbitrarily close to each other but which at  infinite geodesic time produce different limiting polarizations and quantizations. Note, however, that the geodesics that we describe above, which correspond to ``nice" smoothings, have the interesting feature that they are more well adapted to symplectic reduction, since, along $W$, they preserve the holomorphic structure in the directions parallel to the facets of $W$. 
\end{rmk}

\begin{rmk}
A particularly simple case of the above corresponds to the case when the decomposition of $P$ into sub-polytopes is achieved by taking codimension-1 walls such that each wall separates $P$ into two disjoint pieces. If such a wall is described by the hyperplane 
$x_F^\perp = b_1 x_1+\cdots + b_n x_n=0$, for some rational coefficients $b_1, \dots, b_n$, such that $\tilde x = x_F^\perp$ can be made a coordinate in a new  system of coordinates obtained by a $SL_n(\mathbb{Z})$-transformation, then taking 
$$\psi = \psi(\tilde x),$$ such that $\psi''$ has a bump function at $0$, gives a Mabuchi geodesic corresponding to that decompostion of $P$. 
If one has $p$ several such walls, one can take
$$\psi = \psi_1(\tilde x_1)+ \cdots + \psi_p(\tilde x_p),$$
where each $\tilde x_j$ is part of an $SL_n(\mathbb{Z})$-rotated coordinate system and  $\psi_j$ has second derivative given by a bump function at zero.
 An example is given in the figure below for  $\mathbb{CP}^2$.
\begin{figure}[h]
		\centering
			\includegraphics[scale=0.6]{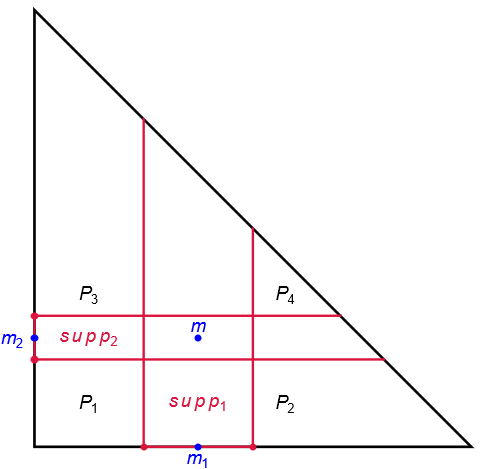}
			\caption{Example for $\mathbb{CP}^2$, with $\psi = \psi_1(x_1)+\psi_2(x_2)$}
   %\label{figure2:CP2}
	\end{figure}\\
\end{rmk}

%%%%%%%%%%%%%%%%%%%%%%%%%%%%%%%%%%%%%%%%%%%
%%%%%%%%%%%%%%%%%%%%%%%%%%%%%%%%%%%%%%%%%%%
%%%%%%%%%%%%%%%%%%%%%%%%%%%%%%%%%%%%%%%%%%%

\bigskip
{}

{\bf Acknowledgements:} The authors were partially supported by the Center for Mathematical Analysis, Geometry and Dynamical Systems  under the projects UIDB/04459/2020 and UIDP/04459/2020. AG 
was the recipient of a fellowship under the project UIDB/04459/2020 CAMGSD.

\printbibliography

\end{document}